\RequirePackage{fix-cm}
\documentclass{svjour3}                    
\smartqed 
\usepackage{mathptmx}     
\journalname{CALCOLO}

\usepackage{amsfonts,latexsym}
\usepackage{epsfig}
\usepackage{multirow}
\usepackage{amsmath}
\usepackage{graphicx}
\usepackage{amsmath}
\usepackage{color}
\usepackage{ulem}
\usepackage{algorithm,algorithmic}

\usepackage{epstopdf} 

\usepackage{graphicx, subfigure}
\DeclareMathOperator*{\argmax}{arg\,max}

\begin{document}
	
	\title{The global extended-rational Arnoldi method \\for matrix function approximation}
	
	
	
	\author{A. H. Bentbib \and M. El Ghomari  \and K. Jbilou   
	}
	
	
	\institute{A. H. Bentbib \at Facult\'e des Sciences et Techniques-Gueliz, Laboratoire de Math\'ematiques Appliqu\'ees et Informatique, Morocco. \email{a.bentbib@uca.ma}          
		\and 
		M. El Ghomari \at Facult\'e des Sciences et Techniques-Gueliz, Laboratoire de Math\'ematiques Appliqu\'ees et Informatique, Morocco.  \email{m.elghomari10@gmail.com} 
		\and 
		K. Jbilou (corresponding author) \at Universit\'e du Littoral, C\^ote d'Opale, batiment H. Poincarr\'e,\\ 
		50 rue F. Buisson,  F-62280 Calais Cedex, France. \email{jbilou@lmpa.univ-littoral.fr} 
	}
	
	\date{Received: date / Accepted: date}

\maketitle

\begin{abstract}
The numerical computation of matrix functions such as  $f(A)V$, where $A$ is an  $n\times n$ large and sparse square matrix, $V$ is an $n \times p$ block with $p\ll n$ and $f$ is a nonlinear matrix function, arises in various applications such as network analysis ($f(t)=exp(t)$ or $f(t)=t^3)$, machine learning $(f(t)=log(t))$, theory of quantum chromodynamics $(f(t)=t^{1/2})$, electronic structure computation, and others. In this work, we propose the use of  global extended-rational Arnoldi method for computing approximations of such expressions. The derived method projects the initial problem onto an global extended-rational Krylov subspace $\mathcal{RK}^{e}_m(A,V)=\text{span}(\{\prod\limits_{i=1}^m(A-s_iI_n)^{-1}V,\ldots,(A-s_1I_n)^{-1}V,V$ $,AV,
\ldots,A^{m-1}V\})$ of a low dimension. An adaptive procedure for the selection of shift parameters $\{s_1,\ldots,s_m\}$ is given. The proposed method is also applied to solve parameter dependent systems. Numerical examples are presented to show the performance of the global extended-rational Arnoldi for these problems.
\end{abstract}


\section{Introduction}\label{Intro}
Let $A\in\mathbb{R}^{n\times n}$ be a large  and sparse matrix, and let  $V\in\mathbb{R}^{n\times p}$ with $1\leq p\ll n$. We are interested in approximating numerically expressions of the form
\begin{equation}\label{If}
\mathcal{I}(f):=f(A)V
\end{equation}
where $f$ is a function that is defined on the convex hull of the spectrum of $A$. The superscript $^T$ denotes transposition. The need to evaluate matrix functions of the forms \eqref{If} arises in various applications such as in network analysis \cite{ESTRADA}, machine learning \cite{NGO}, electronic structure computation \cite{BARONI,SAAD} and the solution of ill-posed problems \cite{FENU,HANSEN}. When the matrix $A$ is a small to meduim size, the matrix function $f(A)$ can be determined by the  spectral factorization of $A$; see  \cite{Higham,HANSEN}, for discussions on several possible definitions of matrix functions. In  many applications, the  matrix $A$ is large that it is impractical to evaluate its spectral factorization. For this case, several projection methods have been developed. These methods consist of projecting the problem \eqref{If} onto a  Krylov subspace with a small dimension. The  projected part  $H$ of $A$  is then used to  evaluate $f(H)$ by determining the spectral factorization of $H$ and then get an approximation of $f(A)V$. In the context of approximating the action of a matrix function $f(A)$ on a some vector $v\in\mathbb{R}^n$, several polynomial methods \cite{Beckermann,DruskinTwo,SAAD1,Hochbruck} based on the standard Arnoldi and Lanczos Krylov methods have been proposed. Druskin and Knizhnerman \cite{DruskinExtended}, have shown that when $f$ cannot be approximated accurately by  polynomials on the spectrum of $A$, then $f(A)v$ cannot be approximated accurately by classical methods. They proposed the extended Krylov method for the symmetric case and the process was generalized to the nonsymmetric matrices by Simoncini in \cite{SimonciniNew}.  This method was applied to approximate the solution of the Sylvester, Riccati and Lyapunov equations \cite{Agougil,Heyouni,SimonciniNew}.

\noindent Another technique for the evaluation of matrix functions is the rational Arnoldi method. This process  was first proposed by Ruhe \cite{Ruhe} in the context of computing the eigenvalues and have been used during the last years for the approximation of matrix functions, see. \cite{Guttel,Pranic,Druskin,DruskinSim,DruskinKnizhnerman,Knizhnerman}. In this paper, we present the global extended-rational Arnoldi method to approximate the matrix function \eqref{If}. The extended-rational Arnoldi method was proposed  and applied  to model reduction by \cite{Abidi1}. As mentioned in \cite{Abidi1}, the extended-rational Krylov subspace \eqref{KmER} is richer than the rational Krylov subspace and represents a generalization of the extended Krylov subspace. We propose an adaptive computation of the shifts $(s_i)$ to generate an $F$-orthonormal basis for \eqref{KmER} in the case where  $f(A)=\{e^{-tA},(A-\sigma I_n)^{-1}\}$ for definite matrix $A$. This procedure is based on a generalization of the procedure used in \cite{Druskin}. In addition, we apply the proposed method to solve parameter dependent systems \eqref{SL} with multiple right hand sides \cite{Gu,SIMONCINI}. These parameter systems have numerous applications in control theory, structural dynamics and time-dependent PDEs; see, \cite{Feriani}. 

\noindent This  paper  is organized as follows. In Section $2$, we give some preliminaries and then  we introduce the global extended-rational Arnoldi process with some properties. Section $3$ describes the application of this process to the approximation of the matrix function given in \eqref{If} and solving the parameter systems. We also propose an adaptive computation of the shifts $(s_i)$. Finally, some numerical experiments that illustrate the quality of the  computed  approximations  are presented in Section $5$.
\section{The global extended-rational Arnoldi method}
\subsection{Preliminaries and notations}
We begin by recalling some notations and definitions that will be used throughout this paper. The Kronecker product satisfies the following properties
\begin{enumerate}
    \item $(A\otimes B)(C\otimes D)=AC\otimes BD.$
    \item $(A\otimes B)^T=A^T\otimes B^T$.
\end{enumerate}

\begin{definition}\cite{BOUYOULI}
Partition the matrices
$M=[M_1,\ldots,M_s]\in\mathbb{R}^{n\times sp}$ and
$N=[N_1,\ldots,N_{l}]\in\mathbb{R}^{n\times lp}$ into block columns
$M_i,N_j\in\mathbb{R}^{n\times p}$, and define the $\diamond$-product of the matrices $M$
and $N$ as
\begin{align}\label{diamond}
M^{T}\diamond N=[\langle N_j,M_i\rangle_F]_{i=1,\ldots,s}^{j=1,\ldots,l}\in{\mathbb R}^{s\times l}.
\end{align}
\end{definition}
The following proposition gives some properties satisfied by the above product.
\begin{proposition}\cite{BellalijJ,BOUYOULI}\label{prop:propertiesdiamond}
Let $A,B,C\in\mathbb{R}^{n\times ps}$, $D\in\mathbb{R}^{n\times n}$, $L\in\mathbb{R}^{p\times p}$, and $\alpha\in\mathbb{R}$. Then we have,
\begin{enumerate}
    \item $(A+B)^T\diamond C=A^T\diamond C+B^T\diamond C$.
    \item $A^T\diamond (B+C)=A^T\diamond B+A^T\diamond C$.
    \item $(\alpha A)^T\diamond C=\alpha(A^T\diamond C)$.
   \item $(A^T\diamond B)^T=B^T\diamond A$.
   \item $(DA)^T\diamond B=A^T\diamond (D^TB)$. 
   \item $A^T\diamond (B(L\otimes I_p))=(A^T\diamond B)L$.
\end{enumerate}
\end{proposition}
\subsection{Description of the process}
Global Krylov subspace techniques were first proposed in \cite{Jbilou} for solving linear systems of equations with multiple right hand sides and also for large-scale  Lyapunov matrix equations. The global extended-rational Krylov subspace was first introduced in \cite{Abidi1} and  it is defined as the  subspace of $\mathbb{R}^{n\times p}$ spanned by the vectors (blocks) 
$$V,AV,\ldots,A^{m-1}V, \;{\rm and}\;  (A-s_1I_n)^{-1}V,(A-s_1I_n)^{-1}(A-s_2 I_n)^{-1}V,\ldots,\prod\limits_{i=1}^{m}(A-s_iI_n)^{-1}V.$$ \\This subspace is denoted by
\begin{equation}\label{KmER}
\scalebox{0.95}{$\mathcal{RK}_m^{e}(A,V)=\text{span}\left\{V,(A-s_1I_n)^{-1}V,\ldots,A^{m-1}V,\prod\limits_{i=1}^m(A-s_iI_n)^{-1}V\right\} \subset \mathbb{R}^{n \times p}$}
\end{equation}
where $\{s_1,\ldots,s_m\}$ are some selected complex parameters all distinct from the eigenvalues of $A$. We notice here that the subspace $\mathcal{RK}_m^{e}(A,V)$ is a subspace of $\mathbb{R}^{n \times p}$ so the vectors are blocks of size $ n\times p$. We assume that all the vectors (blocks) in \eqref{KmER} are linearly independent.
Now, we describe the global extended-rational Arnoldi process to generate an $F$-orthonormal basis $\mathcal{V}_{2m+2}=\{V_1,V_2,\ldots,V_{2m+2}\}$ with $V_i\in\mathbb{R}^{n\times p}$ for the global extended-rational Krylov subspace $\mathcal{RK}_{m+1}^{e}(A,V)$, and we derive some algebraic relations related to this process. The block vector $V_i$'s  are said to be $F$-orthonormal, (with respect to the Frobenius inner product), if
\begin{equation*}
\langle V_j,V_k \rangle_F := {\rm trace}(V_k^TV_j)=
\left\{\begin{array}{cc} 1 & j=k,\\ 0 & j\ne k.\end{array}\right.
\end{equation*}
Based on the Gram-Schmidt orthogonalization process, the first two blocks $V_1$ and $V_2$ are computed via the formulas
\begin{equation}\label{V0}
\begin{array}{ll}
    V_1&=\dfrac{V}{\alpha_{1,1}},\\
    V_2&=\dfrac{\widetilde{V}_2}{\alpha_{2,2}}, \quad \widetilde{V}_2=(A-s_1 I_n)^{-1}V-\alpha_{1,2}V_1 ,
    \end{array}
\end{equation}
where $\alpha_{1,1}=\|V\|_F$, $\alpha_{1,2}=\langle (A-s_1 I_n)^{-1}V,V_1\rangle_F$ and $\alpha_{2,2}=\|\widetilde{V}_2\|$. To compute the block vectors $V_{2j+1}$ and $V_{2j+2}$, for $j=1,\ldots,m-1$, we use the following formulas
\begin{equation}\label{VV}
\begin{array}{ccl}
h_{2j+1,2j-1}V_{2j+1}&=&AV_{2j-1}-\sum\limits_{i=1}^{2j}h_{i,2j-1}V_i=\widetilde{V}_{2j+1}\\
h_{2j+2,2j}V_{2j+2}&=&(A-s_j I_n)^{-1}V_{2j}-\sum\limits_{i=1}^{2j+1} h_{i,2j}V_i=\widetilde{V}_{2j+2}
\end{array}
\end{equation}
where the coefficients $h_{1,2j-1},\ldots,h_{2j,2j-1}$ and $h_{1,2j},\ldots,h_{2j+1,2j}$ are determined so that the vectors satisfy the $F$-orthogonality condition
\[
V_{2j+1}\bot_F V_1,\ldots,V_{2j}\quad \text{and}\quad V_{2j+2}\bot_F V_1,\ldots,V_{2j+1}.
\]
Thus, the coefficients $h_{1,2j-1},\ldots,h_{2j,2j-1}$ and $h_{1,2j},\ldots,h_{2j+1,2j}$ are written as
\begin{align}
h_{i,2j-1}=\langle AV_{2j-1},V_i\rangle_F \quad \text{and} \quad h_{i,2j}=\langle (A-s_j I_n)^{-1}V_{2j},V_i\rangle_F\label{hi}
\end{align}
The coefficients $h_{2j+1,2j-1}$ and $h_{2j+2,2j}$ are such that $\|V_{2j+1}\|_F=1$ and $\|V_{2j+2}\|_F=1$ respectively. Hence, 
\begin{align}
    h_{2j+1,2j-1}=\|\widetilde{V}_{2j+1}\|_F \text{ and } h_{2j+2,2j}=\|\widetilde{V}_{2j+2}\|_F
\end{align}
The global extended-rational Arnoldi algorithm is summarized in the following algorithm (Algorithm \ref{alg:Arnoldi}).

\begin{algorithm}
\caption{The global extended-rational Arnoldi algorithm (GERA)}\label{alg:Arnoldi}
Inputs: Matrix $A$, initial block $V$, and the shifts $\{s_1,\ldots,s_m\}$ .
\begin{enumerate}
\item $\alpha_{1,1}=\|v\|_F;$ $V_1=V/\alpha_{1,1};$
\item $\alpha_{1,2}=\langle(A-s_1 I_n)^{-1}V,V_1\rangle_F;$ $\widetilde{V}_2=(A-s_1 I_n)^{-1}V-\alpha_{1,2}V_1;$
\item $\alpha_{2,2}=\|\widetilde{V}_2\|_F;$ $V_2=\widetilde{V}_2/\alpha_{2,2};$
\item For $j=1:m$
\begin{enumerate}
\item $\widetilde{V}_{2j+1}=AV_{2j-1}$.
\item For $i=1:2j$
\begin{itemize}
\item $h_{i,2j-1}=\langle\widetilde{V}_{2j+1},V_i\rangle_F;$
\item $\widetilde{V}_{2j+1}=\widetilde{V}_{2j+1}-h_{i,2j-1}V_i;$
\item endfor
\end{itemize}
\item $h_{2j+1,2j-1}=\|\widetilde{V}_{2j+1}\|_F;$
\item $V_{2j+1}=\widetilde{V}_{2j+1}/h_{2j+1,2j-1};$
\item $\widetilde{V}_{2j+2}=(A-s_jI_n)^{-1}V_{2j}.$
\item for $i=1:2j+1$
\begin{itemize}
\item $h_{i,2j}=\langle\widetilde{V}_{2j+2},V_i\rangle_F;$
\item $\widetilde{V}_{2j+2}=\widetilde{V}_{2j+2}-h_{i,2j}V_i$
\item endfor
\end{itemize}
\item $h_{2j+2,2j}=\|\widetilde{V}_{2j+2}\|_F;$
\item $V_{2j+2}=\widetilde{V}_{2j+2}/h_{2j+2,2j}$
\item endfor
\end{enumerate}
\end{enumerate}
Output: $F$-Orthonormal basis $\mathcal{V}_{2m+2}=[V_1,\ldots,V_{2m+2}].$
\end{algorithm}

\noindent The set of shifts $\{s_1,\ldots,s_m\}$ is chosen a priori or adaptively during the algorithm. The selection of shifts will be explained later. If all $h_{2j+1,2j-1}$ and $h_{2j+2,2j}$ are numerically nonzero,  then Algorithm \ref{alg:Arnoldi} determines an $F$-orthonormal basis $V_1,\ldots,V_{2m+2}$ $(V_j\in\mathbb{R}^{n\times p})$ of the global extended-rational Krylov subspace $\mathcal{RK}^{e}_{m+1}(A,V)$. Algorithm \ref{alg:Arnoldi} constructs also an upper block Hessenberg matrix $\mathcal{\widetilde{H}}_{2m}=[h_{i,j}]\in\mathbb{R}^{2(m+1)\times 2m}$. We now introduce the $2m\times 2m$ matrix given by
\begin{align}\label{T}
    \mathcal{T}_{2m}=\mathcal{V}_{2m}^T\diamond A\mathcal{V}_{2m}=[t_{i,j}],
\end{align}
where $t_{i,j}=\langle AV_j,V_i\rangle_F,$ $i,j=1,\ldots,2m$. $\mathcal{T}_{2m}$ is the restriction of the matrix $A$ to the extended-rational Krylov subspace $\mathcal{RK}^{e}_m(A,V)$. 
\begin{proposition}
Let the matrix $A\in\mathbb{R}^{n\times n}$, and let $V\in\mathbb{R}^{n\times p}$. The $F$-orthonormal basis $V_1,\ldots,V_{2m+2}$ determined by the recursion formulas \eqref{VV} satisfies, for $j=1,\ldots,m$
\begin{equation}\label{Av}
\begin{array}{ccl}
AV_{2j-1}\text{ and }AV_{2j}&\in&\text{span}\{V_1,\ldots,V_{2j+1}\}
\end{array}
\end{equation}
\end{proposition}
We notice that in the  formulas given by  \eqref{Av}, $\mathcal{T}_{2m}$ is a block upper Hessenberg matrix with $2\times 2$ blocks, since $\langle AV_{2j-1},V_i\rangle_F=0$ and $\langle AV_{2j},V_i\rangle_F=0$ ( for $j=1,\ldots,m$; $i\geq 2j+2$). 
\begin{proposition}
Assume that $m$ steps of Algorithm \ref{alg:Arnoldi} have been run and let $\widetilde{\mathcal{T}}_{2m}=\mathcal{V}^T_{2m+2}\diamond A\mathcal{V}_{2m}$, then we have the following relations
\begin{equation}\label{AT}
\begin{array}{rcl}
A\mathcal{V}_{2m}&=&\mathcal{V}_{2m+2}(\widetilde{\mathcal{T}}_{2m}\otimes I_p)\\
&=&\mathcal{V}_{2m}(\mathcal{T}_{2m}\otimes I_p)+V_{2m+1}(\begin{bmatrix}t_{2m+1,2m-1}\,,&\,t_{2m+1,2m}\end{bmatrix}E^T_m\otimes I_p),
\end{array}
\end{equation}
where the matrix $E_m=[e_{2m-1},e_{2m}]\in\mathbb{R}^{2m\times 2}$ corresponds to the last two columns of the identity matrix $I_{2m}$.
\end{proposition}
\begin{proof}
According to \eqref{Av}, we obtain $A\mathcal{V}_{2m}\in\mathcal{RK}^{e}_{m+1}(A,V)$, then there exists a matrix $T\in\mathbb{R}^{(2m+2)\times 2m}$ such that 
\[
A\mathcal{V}_{2m}=\mathcal{V}_{2m+2}(T\otimes I_p).
\]
Using properties of the $\diamond$-product, we obtain
\begin{align*}
\mathcal{V}^T_{2m+2}\diamond A\mathcal{V}_{2m}&=\mathcal{V}^T_{2m+2}\diamond [\mathcal{V}_{2m+2}(T\otimes I_p)]\\
&=(\mathcal{V}^T_{2m+2}\diamond\mathcal{V}_{2m+2})T=T 
\end{align*}
It follows that $T=\widetilde{\mathcal{T}}_{2m}$. Since $\widetilde{\mathcal{T}}_{2m}$ is an upper block Hessenberg matrix with $2\times 2$ blocks and $t_{2j+2,2j}=\langle Av_{2j},v_{2j+2}\rangle_F=0$ then $\mathcal{V}_{2m+2}(\widetilde{\mathcal{T}}_{2m}\otimes I_p)$ can be decomposed as follows
\[
\mathcal{V}_{2m+2}(\widetilde{\mathcal{T}}_{2m}\otimes I_p)=\mathcal{V}_{2m}\mathcal{T}_{2m}+v_{2m+1}\begin{bmatrix}t_{2m+1,2m-1}\,,&\,t_{2m+1,2m}\end{bmatrix}E^T_m
\]
Which completes the proof.
\end{proof}

\noindent The next proposition gives the entries of $\mathcal{T}_{2m}$ in terms of the
recursion coefficients. This will allow us to compute the entries quite efficiently.  
\begin{proposition}
Let $\widetilde{\mathcal{T}}_{2m}=[t_{:,1},\ldots,t_{:,2m}]$ and $\widetilde{\mathcal{H}}_{2m}=[h_{:,1},\ldots,h_{:,2m}]$ be the upper block Hessenberg matrices where $h_{:,j}$ and $t_{:,j}\in\mathbb{R}^{2m+2}$ are the $j$-th column of $\widetilde{\mathcal{H}}_{2m}$ and $\widetilde{\mathcal{T}}_{2m}$, respectively. Then the odd columns are such that  
\begin{align}\label{t1}
t_{:,2j-1}=h_{:,2j-1} \quad \text{for } j=1,\ldots,m.
\end{align}
The even columns satisfy the following relations
\begin{align}
t_{:,2}&=\dfrac{(\alpha_{1,1}+s_1\alpha_{1,2})e_1+s_1\alpha_{2,2}e_2-\alpha_{1,2}h_{:,1}}{\alpha_{2,2}},\label{t2}\\
\intertext{and for $j=1\ldots,m-1$}
t_{:,2j+2}&=\dfrac{1}{h_{2j+2,2j}}\left[s_jh_{2j+2,2j}e_{2j+2}+e_{2j}-\sum\limits_{i=1}^{2j+1}h_{i,2j}(t_{:,i}-s_j e_i)\right],\label{t3}
\end{align}
where $e_i$ corresponds to the $i$-th column vector of the canonical basis $\mathbb{R}^{2m+2}$ and $\alpha_{1,1}, \alpha_{1,2}$ and $\alpha_{2,2}$ are defined from \eqref{V0}.
\end{proposition}
\begin{proof}
We have $t_{:,2j-1}=\mathcal{V}^T_{2m+2}\diamond AV_{2j-1}$. Therefore, \eqref{t1} follows from the expression of $h_{:,2j-1}$ in \eqref{hi}. Using \eqref{V0}, we obtain 
\[
\alpha_{1,2}V_1+\alpha_{2,2}V_2=\alpha_{1,1}(A-s_1 I_n)^{-1}V_1.
\]
Multiplying this last  equality by $(A-s_1 I_n)$ from the left gives
\[
\alpha_{1,2}(A-s_1 I_n)V_1+\alpha_{2,2}(A-s_1 I_n)v_2=\alpha_{1,1}V_1.
\]
Then the vector $AV_2$ is written as follows
\[
AV_2=\dfrac{1}{\alpha_{2,2}}[(\alpha_{1,1}+s_1\alpha_{1,2})V_1+s_1\alpha_{2,2}V_2-\alpha_{1,2}AV_1].
\]
The relation \eqref{t2} is obtained  by multiplying $AV_2$ by $\mathcal{V}_{2m+2}^T$ with the $\diamond$-product from the left since $t_{:,2}=\mathcal{V}_{2m+2}^T\diamond AV_2$.
The formula \eqref{t3} is obtained from the expression of $AV_{2j+2}$ for $j=1,\ldots,m-1$. Thus, multiplying the second equality in \eqref{VV} by $(A-s_j I_n)$ from the left gives
\[
h_{2j+2,2j}(A-s_j I_n)V_{2j+2}=V_{2j}-\sum\limits_{i=1}^{2j+1}h_{i,2j}(A-s_j I_n)V_i.
\]
Then,
\[
AV_{2j+2}=\dfrac{1}{h_{2j+2,2j}}\left[h_{2j+2,2j}s_jV_{2j+2}+V_{2j}-\sum\limits_{i=1}^{2j+1}h_{i,2j}(Av_i-s_jV_i)\right].
\]
The expression \eqref{t3} is easily obtained  by multiplying $AV_{2j+2}$ by $\mathcal{V}_{2m+2}^T$ with the $\diamond$-product from the left. This concludes the proof of the proposition. 
\end{proof}

\noindent We notice that if $A$ is a symmetric matrix, then the restriction matrix $\mathcal{T}_{2m}$ in \eqref{T} reduces to a symmetric and pentadiagonal matrix with the following nontrivial entries, 
\begin{align*}
t_{i,2j-1}&=h_{i,2j-1} \quad \text{for } i\in\{2j-3,\ldots,2j+1\},\\
t_{1,2}&=\begin{bmatrix} \alpha_{1,1}-(t_{1,1}-s_1)\alpha_{1,2}\end{bmatrix}/\alpha_{2,2},\\
t_{2,2}&=s_1-t_{2,1}\alpha_{1,2}/\alpha_{2,2},\\
t_{3,2}&=-t_{3,1}\alpha_{1,2}/\alpha_{2,2}.
\end{align*}
And for $j=1,\ldots,m-1$
\begin{align*}
t_{2j+1,2j+2}&=( s_jh_{2j+1,2j}-\sum\limits_{i=2j-1}^{2j+1}t_{2j+1,i}h_{i,2j} ) /h_{2j+2,2j},\\
t_{2j+2,2j+2}&=s_j-t_{2j+2,2j+1}h_{2j+1,2j}/h_{2j+2,2j}, {\rm and} \\
t_{2j+3,2j+2}&=-t_{2j+3,2j+1}h_{2j+1,2j}/h_{2j+2,2j}.
\end{align*}
\section{Application to the approximation of matrix functions}
In this section, we will show how to use the global extended-rational Arnoldi algorithm to approximate expression of the form  $f(A)V$. As in \cite{Jbilou}, the global extended-rational Krylov subspace $\mathcal{RK}^{e}_m(A,V)$ defined in \eqref{KmER} can be written as 
\begin{equation}\label{KmD}
    \mathcal{RK}^{e}_m(A,V)=\{X_{2m}\in\mathbb{R}^{n\times p}/X_{2m}=\mathcal{V}_{2m}(\mathcal{Y}_{2m}\otimes I_p)\text{ where } \mathcal{Y}_{2m}\in\mathbb{R}^{2m}\}.
\end{equation}
Then the expression $f(A)V$ can be approximated by
\begin{equation}\label{fem}
    f^{er}_{2m}:=\mathcal{P}_{2m}(f(A)V)=\|V\|_F\mathcal{V}_{2m}(f(\mathcal{T}_{2m})e_1\otimes I_p),
\end{equation}
where $\mathcal{P}_{2m}(X)=\mathcal{V}_{2m}([\mathcal{V}_{2m}^T\diamond X]\otimes I_p)\in\mathcal{RK}^{e}_m(A,V)$ for some $X\in\mathbb{R}^{n\times p}$. \\ The $n \times 2mp$  matrix  $\mathcal{V}_{2m}=[V_1,V_2,\ldots,V_{2m}]$ is the matrix corresponding to the  $F$-orthonormal basis for $\mathcal{RK}^{e}_m(A,V)$ constructed by applying $m$ steps of Algorithm \ref{alg:Arnoldi} to the pair $(A,V)$. $\mathcal{T}_{2m}$ is the projection matrix defined by \eqref{T} and $e_1$ corresponds to the first column of the identity matrix $I_{2m}$. 
\begin{lemma}[Exactness] Let $\mathcal{V}_{2m}\in\mathbb{R}^{n\times 2mp}$ be the matrix generated by Algorithm \ref{alg:Arnoldi} and let $\mathcal{T}_{2m}$ the matrix as defined by \eqref{T}. Then for any rational function $\Tilde{r}_{2m}\in\Pi_{2m}/q_m$ we have
\begin{align}
    \mathcal{P}_{2m}(\Tilde{r}_{2m}(A)V)=\|V\|_F\mathcal{V}_{2m}[(\Tilde{r}_{2m}(\mathcal{T}_{2m})e_1)\otimes I_p]
\end{align}
In particular, if $r_{2m}\in\Pi_{2m-1}/q_m$ then the global extended-rational Arnoldi approximation is exact, i.e., we have 
\begin{align}\label{rm}
    r_{2m}(A)V=\|V\|_F\mathcal{V}_{2m}[(r_{2m}(\mathcal{T}_{2m})e_1)\otimes I_p]
\end{align}
where $\Pi_{2m}$ denotes the set of polynomials of degree at most $2m$ and  $q_m$ is the  polynomial of degree $m$, whose roots are the components of $\{s_1,\ldots,s_m\}$, i.e.,  $
q_m(z)=(z-s_1)\ldots(z-s_m)$. 
\end{lemma}

\begin{proof}
Following the idea in \cite[Lemma 3.1]{Guttel}, we consider $q=q_m(A)^{-1}V$ and we first show by induction that
\begin{align}\label{Ajq}
\mathcal{P}_{2m}(A^jq)=\mathcal{V}_{2m}(\mathcal{T}_{2m}^j(\mathcal{V}_{2m}^T\diamond q)\otimes I_p) \quad \text{for } j=0,\ldots,2m.
\end{align}
Assertion \eqref{Ajq} is obviously true for $j=0$. Assume that it is true for some $j<m$. Then by the definition of a extended-rational Krylov space we have $\mathcal{P}_{2m}(A^jq)=A^jq$, and therefore 
\begin{align*}
\mathcal{P}_{2m}(A^{j+1}q)&=\mathcal{P}_{2m}(A\mathcal{P}_{2m}(A^jq))=\mathcal{P}_{2m}(A\mathcal{V}_{2m}[\mathcal{T}^j_{2m}(\mathcal{V}_{2m}^T\diamond q)\otimes I_p])\\
&=\mathcal{V}_{2m}([\mathcal{V}_{2m}^T\diamond (A\mathcal{V}_{2m}[(\mathcal{T}_{2m}^j(\mathcal{V}^T_{2m}\diamond q))\otimes I_p])\otimes I_p)\\
\intertext{Using the properties of the $\diamond$-product, we obtain}
&=\mathcal{V}_{2m}[(\mathcal{V}_{2m}^T\diamond A\mathcal{V}_{2m})(\mathcal{T}_{2m}^j(\mathcal{V}_{2m}^T\diamond q))\otimes I_p]=\mathcal{V}_{2m}(\mathcal{T}_{2m}^{j+1}(\mathcal{V}_{2m}^T\diamond q)\otimes I_p),
\end{align*}
which establishes \eqref{Ajq}. By linearity, we obtain 
\[
V=q_m(A)q=\mathcal{V}_{2m}(q_m(\mathcal{T}_{2m})(\mathcal{V}_{2m}^T\diamond q)\otimes I_p).
\]
Using the properties of the $\diamond$-product, we obtain
\[
\mathcal{V}_{2m}^T\diamond q=\|V\|_Fq_m^{-1}(\mathcal{T}_{2m})e_1.
\]
Replacing $\mathcal{V}_{2m}^T\diamond q$ in \eqref{Ajq} completes the proof. 
\end{proof}

\noindent We consider a convex compact set $\Lambda\subset\mathbb{R}$ and we define $\mathbb{A}(\Lambda)$ as the set of analytic functions  in a neighborhood of $\Lambda$ equipped with the uniform norm $\|.\|_{L^{\infty}(\Lambda)}$ .  $\mathbb{W}(A):=\{x^TAx\,:\,x\in\mathbb{R}^n\,,\,\|x\|=1\}$ will denote the convex hull. 
In  \cite{CROUZEIX}, it was shown  that there exists a universal constant $C=1+\sqrt{2}$ such that 
\begin{align*}\label{crou}
    \|f(B)\|\leq C \|f\|_{L^{\infty}(\Lambda)}, \quad \forall f\in\mathbb{A}(\Lambda),
\end{align*}
where $B\in\mathbb{R}^{n\times n}$, with $\mathbb{W}(B)\subseteq\Lambda.$ Based on this inequality, the following result gives an upper bound for the norm of the error $f(A)V-f^{er}_{2m}$ where $f^{er}_{2m}$ is the approximation given by (\ref{fem}).
\begin{corollary}
We assume that $\mathbb{W}(A)\cup \mathbb{W}(\mathcal{T}_{2m})\subseteq \Lambda$, and let $f\in\mathbb{A}(\Lambda)$. Then the global extended-rational Arnoldi approximation $f^{er}_{2m}$ defined by \eqref{fem} satisfies
\[
\| f(A)V-f^{er}_{2m}\|_F\leq 2C\|V\|_F\min\limits_{r_{2m}\in\Pi_{2m-1}/q_m}\|f-r_{2m}\|_{L^{\infty}(\Lambda)}
\]
\end{corollary}
\begin{proof}
According to \eqref{rm}, we have $r_{2m}(A)V=\|V\|_F\mathcal{V}_{2m}[(r_{2m}(\mathcal{T}_{2m})e_1)\otimes I_p]$ for every rational function $r_{2m}\in\Pi_{2m-1}/q_m$. Thus, 
\begin{align*}
    \|f(A)V-f^{er}_{2m}\|_F&=\|f(A)V-\|V\|_F\mathcal{V}_{2m}(f(\mathcal{T}_{2m})e_1\otimes I_p)-r_{2m}(A)V+\|V\|_F\mathcal{V}_{2m}(r_{2m}(\mathcal{T}_{2m})e_1\otimes I_p)\|_F\\
    &\leq \|V\|_F(\|f(A)-r_{2m}(A)\|+\|\mathcal{V}_{2m}[(f(\mathcal{T}_{2m}-r_{2m}(\mathcal{T}_{2m}))e_1\times I_p]\|_F\\
    &\leq \|V\|_F(\|f(A)-r_{2m}(A)\|+\|f(\mathcal{T}_{2m})-r_{2m}(\mathcal{T}_{2m})\|)\\
    &\leq2C\|V\|_F\|f-r_{2m}\|_{L^{\infty}(\Lambda)}.
\end{align*}
Which completes the proof.
\end{proof}

\subsection{Shifted linear systems}
We consider the solution of the parameterized nonsingular linear systems with multiple right hand sides
\begin{equation}\label{SL}
    (A-\sigma I_n)X=B,
\end{equation}
which needs to  be solved for many values of $\sigma$, where $B\in\mathbb{R}^{n\times p}$. The solution $X=X(\sigma)$ may be written as $X=(A-\sigma I)^{-1}B\equiv f(A)B$, with $f(z)=(z-\sigma)^{-1}$ is the resolvant function. Then the approximate solutions $X_{2m}=X_{2m}(\sigma)\in\mathbb{R}^{n\times p}$ generated by the global extended-rational algorithm to the pair $(A,R_0)$ are obtained as follows
\begin{equation*}
    Z_{2m}(\sigma)=X_{2m}(\sigma)-X_0(\sigma)\in \mathcal{RK}^{e}_{m}(A,R_0)
\end{equation*}
where $R_0=R_0(\sigma)=B-(A-\sigma I)^{-1}X_0(\sigma)$ are the residual block vectors associated to initial guess $X_0(\sigma)$. By \eqref{KmD} $Z_{2m}(\sigma)=\mathcal{V}_{2m}(Y_{2m}(\sigma)\otimes I_p)$ where $Y_{2m}(\sigma)\in\mathbb{R}^{2m}$ is determined such that the new residual $R_{2m}(\sigma)=B-(A-\sigma I_n)X_{2m}$ associated  to $X_{2m}$ is $F$-orthogonal to $\mathcal{RK}^{e}_{2m}(A,R_0)$. This yields 
\begin{equation}\label{OC}
    X_{2m}(\sigma)=X_0+\mathcal{V}_{2m}(Y_{2m}(\sigma)\otimes I_p) \quad \text{and} \quad \mathcal{V}_{2m}^T\diamond R_{2m}(\sigma)=0
\end{equation}
Using \eqref{OC} relations and the following decomposition 
\begin{equation}\label{As}
    (A-\sigma I_n)\mathcal{V}_{2m}=\mathcal{V}_{2m}[(\mathcal{T}_{2m}-\sigma I_{2m})\otimes I_p]+V_{2m+1}\big(\begin{bmatrix}t_{2m+1,2m-1}\,,&\,t_{2m+1,2m}\end{bmatrix}E^T_m\otimes I_p\big),
\end{equation}
the reduced linear system can be written as 
\begin{equation}\label{RLS}
    (\mathcal{T}_{2m}-\sigma I_{2m})Y_{2m}(\sigma)=\|R_0\|_F e_1,
\end{equation}
then the approximate solution will be \[X_{2m}=\|R_0\|_F\mathcal{V}_{2m}((\mathcal{T}_{2m}-\sigma I_{2m})^{-1}e_1\otimes I_p).\]
This equality is equivalent to \eqref{fem} when $f(A)=(A-\sigma I_n)^{-1}$ corresponds to the resolvent function. In order to find a good choice of shift parameters $\{s_1,\ldots,s_m\}$ in the Algorithm \ref{alg:Arnoldi}, we consider the following function 
\begin{equation}\label{fskel}
    f_{2m,m}(\lambda,s)=f_{\lambda_1\ldots,\lambda_{2m},s_1,\ldots,s_m}(\lambda,s)=f_{m,m}(\lambda,s)-\dfrac{1}{\lambda-s}\bigg[\dfrac{g_{2m}(\lambda)}{g_{2m}(s)}-\dfrac{g_{m}(\lambda)}{g_{m}(s)}\bigg]
\end{equation}
where $f_{m,m}=f_{\lambda_1\ldots,\lambda_{m},s_1,\ldots,s_m}$ corresponds to the so-called skeleton approximation introduced in the study of Tyrtyshnikov \cite{Tyrtyshnikov}, i.e., 
\[
f_{m,m}(\lambda,s):=\bigg[\dfrac{1}{\lambda-s_1}\cdots\dfrac{1}{\lambda-s_m}\bigg]
M^{-1}\begin{bmatrix}
 \dfrac{1}{\lambda_1-s}\\\vdots\\ \dfrac{1}{\lambda_m-s}
\end{bmatrix},\quad M_{i,j}=\dfrac{1}{\lambda_i-s_j}
\]
$\lambda_1,\ldots,\lambda_{2m}$ are the eigenvalues of $\mathcal{T}_{2m}$, and
\[
g_m(z)=\dfrac{(z-\lambda_1)\ldots(z-\lambda_m)}{(z-s_1)\ldots(z-s_m)} \quad g_{2m}(z)=\dfrac{(z-\lambda_1)\ldots(z-\lambda_{2m})}{(z-s_1)\ldots(z-s_m)}
\]
\begin{proposition}
The function $f_{2m,m}$ defined in \eqref{fskel} is an $[(2m-1)|m]$ rational function of the first variable $\lambda$, and an $[(2m-1)|2m]$ rational function of the second variable $s$ interpolating $(\lambda-s)^{-1}$ at $\lambda=\lambda_1,\ldots,\lambda_{2m}$ and $s=s_1,\ldots,s_m$. Moreover, the relative error of this interpolation is 
\begin{equation*}
    \epsilon(\lambda,s)=1-(\lambda-s)f_{2m,m}(\lambda,s)=\dfrac{g_{2m}(\lambda)}{g_{2m}(s)}.
\end{equation*}
\end{proposition}
\begin{proof}
The rational function $f_{2m,m}$ can be expressed as 
\[
f_{2m,m}(\lambda,s)=f_{m,m}(\lambda,s)-\dfrac{1}{\lambda-s}\dfrac{\psi(\lambda,s)}{(\lambda-s_1)\ldots(\lambda-s_m)(s-\lambda_1)\ldots(s-\lambda_{2m})}
\]
where 
\[
\psi(\lambda,s)=(\lambda-\lambda_1)\ldots(\lambda-\lambda_m)(s-s_1)\ldots(s-s_m)[(\lambda-\lambda_{m+1})\ldots(\lambda-\lambda_{2m}-(s-\lambda_{m+1})\ldots(s-\lambda_{2m})]
\]
Observe that $\psi(\lambda,s)$ is divisible by $(\lambda-s)$, then there exists a function $\phi$ such that $\psi(\lambda,s)=(\lambda-s)\phi(\lambda,s)$. Moreover, $\phi$ is a polynomial function of degree $2m-1$ of each variable. Then $f_{2m,m}$ simplifies to 
\begin{align*}
    f_{2m,m}(\lambda,s)&=f_{m,m}(\lambda,s)-\dfrac{\phi(\lambda,s)}{(\lambda-s_1)\ldots(\lambda-s_m)(s-\lambda_1)\ldots(s-\lambda_{2m})}\\
    &=\dfrac{f_{m,m}(\lambda,s)(\lambda-s_1)\ldots(\lambda-s_m)(s-\lambda_1)\ldots(s-\lambda_{2m})-\phi(\lambda,s)}{(\lambda-s_1)\ldots(\lambda-s_m)(s-\lambda_1)\ldots(s-\lambda_{2m})}
\end{align*}
The relative error \cite{DruskinKnizhnerman} of the skeleton approximation $f_{m,m}$ is given by
\begin{align}\label{ErrSke}
    1-(\lambda-s)f_{m,m}(\lambda,s)=\dfrac{g_m(\lambda)}{g_m(s)}.
\end{align}
According to this equality, we can show that there exists a polynomial function $\varphi$, of degree $m-1$ (of each variable) such that 
\[
f_{m,m}(\lambda,s)(\lambda-s_1)\ldots(\lambda-s_m)(s-\lambda_1)\ldots(s-\lambda_{2m})=\varphi(\lambda,s)(s-\lambda_{m+1})\ldots(s-\lambda_{2m}).
\]
Thus, 
\[
f_{2m,m}(\lambda,s)=\dfrac{\varphi(\lambda,s)(s-\lambda_{m+1})\ldots(s-\lambda_{2m})-\phi(\lambda,s)}{(\lambda-s_1)\ldots(\lambda-s_m)(s-\lambda_1)\ldots(s-\lambda_{2m})}.
\]
Which shows that $f_{2m,m}$ is an $[(2m-1)|m]$ of the first variable $\lambda$ and an $[(2m-1)|2m]$ of the second variable $s$. It is clear that $f_{2m,m}$ interpolates $(\lambda-s)^{-1}$ at $\lambda=\lambda_1,\ldots,\lambda_m$ and $s=s_1,\ldots,s_m$. For $\lambda=\lambda_{m+1},\ldots,\lambda_{2m}$, we have
\begin{align*}
    f_{2m,m}(\lambda_i,s)&=f_{m,m}(\lambda_i,s)-\dfrac{1}{\lambda_i-s}\bigg[\dfrac{g_{2m}(\lambda_i)}{g_{2m}(s)}-\dfrac{g_{m}(\lambda_i)}{g_{m}(s)}\bigg].\\
    &=f_{m,m}(\lambda_i,s)+\dfrac{1}{\lambda_i-s}\dfrac{g_{m}(\lambda_i)}{g_{m}(s)}.\\
    \intertext{Using the relation error equation \eqref{ErrSke}, we get}
    &=\dfrac{1}{\lambda_i-s}\bigg[1-\dfrac{g_{m}(\lambda_i)}{g_{m}(s)}\bigg]+\dfrac{1}{\lambda_i-s}\dfrac{g_{m}(\lambda_i)}{g_{m}(s)}=\dfrac{1}{\lambda_i-s}.
\end{align*}
Which means that $f_{2m,m}$ interpolates $(\lambda-s)^{-1}$ at $\lambda=\lambda_1,\ldots,\lambda_{2m}$ and $s=s_1,\ldots,s_m$. The relative error is 
\begin{align*}
    1-(\lambda-s)f_{2m,m}(\lambda,s)&=1-(\lambda-s)\bigg[f_{m,m}(\lambda,s)-\dfrac{1}{\lambda-s}\bigg(\dfrac{g_{2m}(\lambda)}{g_{2m}(s)}-\dfrac{g_{m}(\lambda)}{g_{m}(s)}\bigg)\bigg].\\
    &=1-(\lambda-s)f_{m,m}(\lambda,s)+\dfrac{g_{2m}(\lambda)}{g_{2m}(s)}-\dfrac{g_{m}(\lambda)}{g_{m}(s)}.\\
   \intertext{Using the relation error equation \eqref{ErrSke}, we conclude that}
   1-(\lambda-s)f_{2m,m}(\lambda,s)=\dfrac{g_{2m}(\lambda)}{g_{2m}(s)}.
\end{align*}\qed
\end{proof}

\noindent Using the same techniques as in \cite{Knizhnerman}, we can show that 
\begin{equation*}
    Z_{2m}(\sigma)=f_{2m,m}(A,\sigma)R_0. 
\end{equation*}
By this equality, the residual $R_{2m}(\sigma)$ can be expressed as 
\begin{equation}\label{ErrSk}
    R_{2m}(\sigma)=R_0-(A-\sigma I_n)f_{2m,m}(A,\sigma)R_0=\dfrac{g_{2m}(A)R_0}{g_{2m}(\sigma)}
\end{equation}
From \cite[Proposition 2]{Datta}, the characteristic polynomial of $\mathcal{T}_{2m}$ minimizes $\| p(A)R_0\|_F$ over all monic polynomial of degree $2m$, so that the numerator in \eqref{ErrSk} satisfies
\begin{equation*}
    \|g_{2m}(A)R_0\|_F=\min\limits_{\lambda_1,\ldots,\lambda_{2m}}\|(A-\lambda_1 I_n)\ldots(A-\lambda_{2m} I_n)(A-s_1 I_n)^{-1}\ldots(A-s_m I_n)^{-1}R_0\|_F
\end{equation*}
With this result, and \eqref{ErrSk} equation, the next shift parameter $s_{m+1}$ is selected as
\begin{equation*}
   s_{m+1}=\argmax\limits_{\sigma\in \Sigma} \dfrac{1}{|g_{2m}(\sigma)|}
\end{equation*}
where $\Sigma$ is the set of the shifts associated to the parameterized linear systems \eqref{SL}. The following result on the norm of the residual $R_{2m}(\sigma)$ allows us to stop the iterations without having to compute matrix products with
the large matrix $A$.
\begin{theorem}
Let $Y_{2m}(\sigma)$ be the exact solution of the reduced linear system \eqref{RLS} and let  $X_{2m}(\sigma)$ be the approximate solution of linear system \eqref{SL} after $m$ iterations of the extended rational global Arnoldi algorithm. Then the residual $R_{2m}(\sigma)$ satisfies 
\begin{equation}\label{Resi}
    \|R_{2m}(\sigma)\|_F=\|R_{0}(\sigma)\|_F\|\tau_{2m}E^T_m(\mathcal{T}_{2m}-\sigma I_{2m})^{-1}e_1\|_F
\end{equation}
where $\tau_{2m}=\begin{bmatrix}t_{2m+1,2m-1}\,,&\,t_{2m+1,2m}\end{bmatrix}$
\end{theorem}
\begin{proof}
\begin{align}
    R_{2m}(\sigma)&=B-(A-\sigma I_n)X_{2m}=B-(A-\sigma I_n)(X_0+\mathcal{V}_{2m}(Y_{2m}\otimes I_p)).\nonumber\\
    \intertext{Using \eqref{As} decomposition, we obtain }
    &=R_0-[\mathcal{V}_{2m}((\mathcal{T}_{2m}-\sigma I_{2m})\otimes I_p+V_{2m+1}(\tau_{2m}E^T_m\otimes I_p)](Y_{2m}\otimes I_p).\nonumber\\
    &=R_0-\mathcal{V}_{2m}[(\mathcal{T}_{2m}-\sigma I_{2m})Y_{2m}\otimes I_p]-V_{2m+1}(\tau_{2m}E^T_m\otimes I_p).\nonumber\\
    &=R_0-\mathcal{V}_{2m}(\|R_{0}(\sigma)\|_Fe_1\otimes I_p)-V_{2m+1}(\tau_{2m}E^T_m\otimes I_p).\nonumber\\
    &=V_{2m+1}(\tau_{2m}E^T_m\otimes I_p).\label{RRR}
\end{align}
\end{proof}

\noindent As $m$ increases, the column of block vectors that must be stored increases. As in \cite{SAAD,SIMONCINI}, we can restart the algorithm every some fixed $m$ steps. According to \eqref{RRR}, we observe that the residuals $R_{2m}(\sigma)$ are colinear to the block vector $V_{2m+1}$. Then it is possible to restart with $V_{2m+1}$ and $\beta_0(\sigma)=\text{trace}(V_{2m+1}^TR_{2m}(\sigma))$ see, Algorithm  \ref{alg:RestartedSLS} [line $8$].
\begin{algorithm}
\caption{Restarted shifted linear system algorithm}\label{alg:RestartedSLS}
Input: Matrix $A$, block vector $B$, $\Sigma$ the set of shifts, $\epsilon$ a desired tolerance.
\begin{enumerate}
\item Set $\beta_0(\sigma)=\|B\|_F,$ $V_1=B/\beta_0(\sigma)$ and $\Sigma_c=\emptyset$.
\item Compute $\mathcal{V}_{2m}$ and $\mathcal{T}_{2m}$  using the global extended-rational Arnoldi algorithm \ref{alg:Arnoldi}.
\item Solve the reduced shifted linear system  $(\mathcal{T}_{2m}-\sigma I_{2m})Y_{2m}=\beta_0(\sigma)e_1,$ for $\sigma\in\Sigma\backslash\Sigma_c$.
\item Compute $\|R_{2m}(\sigma)\|_F$ using \eqref{Resi}, for $\sigma\in\Sigma\backslash\Sigma_c$.
\item Compute $X_{2m}(\sigma)=X_{2m}(\sigma)+\mathbb{V}_{2m}(Y_{2m}(\sigma)\otimes I_p),$ for $\sigma\in\Sigma\backslash \Sigma_c$. 
\item Select the new $\sigma\in\Sigma\backslash\Sigma_c$ such that $\|R_{2m}(\sigma)\|_F<\epsilon$. Update set $\Sigma_c$ of converged shifted systems.
\item if $\Sigma\backslash\Sigma_c=\emptyset$ Stop
\item else Set $V_1=V_{2m+1}$ and $\beta_0(\sigma)=\text{trace}(V_1^TR_{2m}(\sigma)),$ for $\sigma\in\Sigma\backslash \Sigma_c.$
\end{enumerate}
\end{algorithm}

\subsection{Application to the approximation of $e^{-tA}V,$ $t>0$}
In this subsection, we consider the computation of $U(t)=e^{-tA}V$ where $t>0$, for a given large and sparse matrix $A$ and a given block vector $V$ of size $n \times p$. Based on the $F$-orthonormal basis defining the matrix  $\mathcal{V}_{2m}$ generated by the global extended-rational algorithm, the expression of $U(t)$ can be approximated as
\begin{align}\label{exp}
U_{2m}(t)=\|V\|_F\mathcal{V}_{2m}((e^{-t\mathcal{T}_{2m}}e_1)\otimes I_p).
\end{align}
Indeed, the inverse Laplace representation of the resolvent function is written as follows
\begin{align*}
    e^{-tA}V=\dfrac{1}{2\pi i}\int^{i\infty}_{-i\infty}e^{ts}(A+sI_n)^{-1}Vds.
\end{align*} We have seen that the approximation of $(A+sI)^{-1}V$ is 
\[
\mathcal{P}_{2m}((A+sI)^{-1}V)=\|V\|_F\mathcal{V}_{2m}((\mathcal{T}_{2m}+sI)^{-1}e_1\otimes I_p).
\]
Then,
\begin{align*}
    \mathcal{P}_{2m}(e^{-tA}V)&=\dfrac{1}{2\pi i}\int^{i\infty}_{-i\infty}e^{ts}\mathcal{P}_{2m}((A+sI)^{-1}V)ds.\\
    &=\dfrac{1}{2\pi i}\int^{i\infty}_{-i\infty}e^{ts}\|V\|_F\mathcal{V}_{2m}((\mathcal{T}_{2m}+sI)^{-1}e_1\otimes I_p)ds.\\
    &=\|V\|_F\mathcal{V}_{2m}\bigg(\dfrac{1}{2\pi i}\int^{i\infty}_{-i\infty}e^{ts}((\mathcal{T}_{2m}+sI)^{-1}e_1)ds\otimes I_p\bigg).\\
    &=\|V\|_F\mathcal{V}_{2m}(e^{-t\mathcal{T}_{2m}}e_1\otimes I_p).
\end{align*}
We recall that $U(t)=e^{-tA}V$ is the solution of the differential problem 
\begin{equation}\label{Ut}
\begin{array}{cccc}
    U^{'}(t)&=&-AU(t),\quad &t>0\\
    U(0)&=&V,\quad &U(\infty)=0.
    \end{array}
\end{equation}
The residual with respect to this ODE is given by 
\begin{align*}
    R_{2m}(t)=U^{'}_{2m}(t)+AU_{2m}(t).
\end{align*}
Following the idea in \cite{SAAD1}, and by the first equation in \eqref{AT}, the residual is given by the quantity
\begin{align*}
    R_{2m}(t)=\mathcal{V}_{2m}(\tau_mE^T_me^{-t\mathcal{T}_{2m}}e_1\otimes I_p)
\end{align*}
Applying the first result of \cite[Lemma $1$]{Jbilou2} to $R_{2m}(t)$, we obtain the following stopping criterion \begin{align}\label{Rm}
    \|R_{2m}(t)\|_F=\|\tau_mE^T_me^{-t\mathcal{T}_{2m}}e_1\|_F
\end{align}
where $\tau_m=\begin{bmatrix}t_{2m+1,2m-1}\,,&\,t_{2m+1,2m}\end{bmatrix}$ and $E_m$ are defined in \eqref{AT}.\\ 
The following result which concerns the approximation of the exponential will be the key to find a good choice of the shift parameters independently on the parameter $t$. This result is obtained by following some ideas in \cite{DruskinKnizhnerman}.
\begin{theorem}
We assume that $A$ is a positive matrix with spectrum contained on  $[0,\infty)$. Then we have the following error bound
\begin{align}\label{AA}
    \sup\limits_{t\in[0,\infty[}\|e^{-tA}V-U_{2m}(t)\|_F\leq\|g_{2m}(A)V\|_F\max\limits_{s\in i\mathbb{R}}\dfrac{1}{|g_{2m}(-s)|}
\end{align}
\end{theorem}
\begin{proof}
\begin{align*}
    e^{-tA}V-U_{2m}(t)&=e^{-tA}V-\|V\|_F\mathcal{V}_{2m}((e^{-t\mathcal{T}_{2m}}e_1)\otimes I_p)\\
    &=\dfrac{1}{2\pi i}\int^{i\infty}_{-i\infty}e^{ts}(A+sI_n)^{-1}[V-\|V\|_F(A+sI_n)\mathcal{V}_{2m}[(\mathcal{T}_{2m}+sI)^{-1}e_1\otimes I_p]]ds\\
    \intertext{We have\quad $\|V\|_F\mathcal{V}_{2m}[(\mathcal{T}_{2m}+sI_{2m})^{-1}e_1\otimes I_p]=f_{2m,m}(A,-s)$, then}
      &=\dfrac{1}{2\pi i}\int^{i\infty}_{-i\infty}e^{ts}(A+sI_n)^{-1}[V-(A+sI_n)f_{2m,m}(A,-s)V]ds\\
    &=\dfrac{1}{2\pi i}\int^{i\infty}_{-i\infty}e^{ts}(A+sI_n)^{-1}\dfrac{g_{2m}(A)V}{g_{2m}(-s)}ds\\
    \intertext{then}
    \|e^{-tA}V-U_{2m}(t)\|_F&\leq \dfrac{\|g_{2m}(A)V\|_F}{\min\limits_{s\in i\mathbb{R}}|g_{2m}(-s)|}\sup\limits_{\lambda\in[0,\infty)}|h(\lambda,t)|
    \end{align*}
    where $h(\lambda,t)=\dfrac{1}{2\pi i}\int^{i\infty}_{-i\infty}e^{st}(\lambda+s)^{-1}ds$. We observe that $h(\lambda,t)$ corresponds to the inverse Laplace transform  of $1/(\lambda+s)$, then $h(\lambda,t)=e^{-\lambda t}$. Which leads to obtain 
    \[
    \|e^{-tA}V-U_{2m}(t)\|_F\leq \dfrac{\|g_{2m}(A)V\|_F}{\min\limits_{s\in i\mathbb{R}}|g_{2m}(-s)|}
    \]
This shows \eqref{AA} since $\max\limits_{s\in i\mathbb{R}}\{1/|g_{2m}(-s)|\}=1/\min\limits_{s\in i\mathbb{R}}\{|g_{2m}(-s)|\}$ .
\end{proof}

\noindent As for rational Arnoldi approximation, and when working with bounded positive definite matrix $A$, Druskin et al.  \cite{Druskin} showed that real shifts on the spectrum of $A$ can also reach the minimum in \eqref{AA} inequality. We observe that the function $g^{-1}_{2m}(-s)$  has poles at $s=-\lambda_i\in[-\lambda_{max},-\lambda_{min}]$, $i=1,\ldots,2m$. Following the same techniques in \cite[Proposition 2.3]{Druskin}, we can show that all the extrema of $|g_{2m}^{-1}(-s)|$ are ripples (local maxima of $|g_{2m}^{-1}(-s)|$) located only between the interpolation points $\{\lambda_i\}_{i=1}^{2m}$, such that there is one and only one ripple between two adjacent interpolation points. With the result, the next shift parameter $s_{m+1}$ is selected as the corresponding argument of the maximum of $|g^{-1}_{2m}(-\lambda_{max})|$, $|g^{-1}_{2m}(-\lambda_{min})|$ and the $2m-1$ local maxima between the interpolation points. The algorithm for constructing the next shift parameter is given in Algorithm \ref{alg:shiftpara}. Algorithm \ref{alg:exponential} describes how approximations of $e^{-tA}V$ are computed by the adaptive global extended-rational method.

\begin{algorithm}
\caption{The procedure for selecting the shift parameters of exponential function}\label{alg:shiftpara}
Inputs: $\{\lambda_i\}_{i=1}^{2m}$ the set of interpolation points (the eigenvalues of $\mathcal{T}_{2m}$).
\begin{enumerate}
    \item Estimate $\lambda_{min}$ and $\lambda_{max}$.
    \item For $j=1:2m-1$
    \begin{enumerate}
        \item $\mu_j=\argmax\limits_{s\in]\lambda_j,\lambda_{j+1}[}\dfrac{1}{|g_{2m}(-s)|}$
        \item endfor
    \end{enumerate}
    \item  $s_{m+1}=\argmax\limits_{\mu_1,\ldots,\mu_{2m-1},\lambda_{min},\lambda_{max}}\bigg\{\dfrac{1}{|g_{2m}(-\mu_j)|},\dfrac{1}{|g_{2m}(-\lambda_{min})|},\dfrac{1}{|g_{2m}(-\lambda_{max})|}\bigg\}$
\end{enumerate}
\end{algorithm}
\begin{algorithm}
\caption{Approximation of $e^{-tA}V$ by the adaptive global extended-rational method (AGER)}\label{alg:exponential}
Inputs:  Matrix $A$, initial block $V$.
\begin{enumerate}
\item Choose a tolerance $tol>0$, a maximum number of $itermax$ iterations.
\item Estimate $\lambda_{min}$ and set $s_1=\lambda_{min}$.
\item For $m=1:  itermax$
\begin{enumerate}
\item Compute $\mathcal{V}_{2m}$ and $\mathcal{T}_{2m}$ using the global extended-rational Arnoldi algorithm \ref{alg:Arnoldi}.
\item  Compute $Y_{2m}=e^{-t\mathcal{T}_{2m}}e_1$, and compute $\|R_{2m}\|_F$ given by \eqref{Rm}.
\item if $\|R_{2m}\|_F\leq tol$, stop,
\item else Find $s_{m+1}$ by using Algorithm \ref{alg:shiftpara}.
\item endfor
The approximate solution $U_{2m}$ given by \eqref{exp}. 
\end{enumerate}
\end{enumerate}
\end{algorithm}




\newpage

\section{Numerical experiments}
In this section, we give some numerical results to show the performance of the global extended-rational Arnoldi method.  All experiments were  carried out with  MATLAB $R2015a$ on a computer with an Intel Core i-3 processor and 3.89 GBytes of RAM. The computations were  done with about $15$ significant decimal digits. The proposed method is applied to the approximation of $f(A)V$ given in \eqref{If}, and to solve the shifted linear systems \eqref{SL} with multiple right hand sides for many values of $\sigma$. 
\subsection{Examples for the shifted linear systems}
In this subsection, we present some results of solving  shifted linear systems of the form  \eqref{SL}. We compare the results obtained by the restarted global extended-rational Arnoldi (resGERA), the restarted  global extended Arnoldi (resGEA) and the restarted global FOM (resGFOM) methods. The right hand side $B$ was chosen randomly with entries uniformly distributed on $[0,1]$. The shifts $\sigma$ are taken to be  values uniformly distributed in the interval $[-5,0]$. In Example 1 and Example 2, the stopping criterion used for Algorithm \ref{alg:RestartedSLS} was $\|R_{2m}(\sigma)\|_F\leq 2\cdot 10^{-12}$ and the initial guess was zero. The dimension of the subspaces was chosen to be $m=10,20$.\\

\noindent {\bf Example 1}
In this experiment, we consider the nonsymmetric matrices $A_1$ and $A_2$ given in \cite{SIMONCINI} and \cite{ABIDI}, respectively. These matrices were obtained from the centered finite difference discretization (CFDD) of the elliptic  operators $\mathcal{L}_{1}(u)$ and $\mathcal{L}_{2}(u)$,  respectively, 
\begin{equation}\label{Op}
\begin{array}{ll}
\mathcal{L}_1(u)&=-\Delta u+50(x+y)u_x+50(x+y)u_y.\\
\mathcal{L}_2(u)&=-\Delta u+sin(xy)u_x+e^xu_y+(x+y)u.
\end{array}
\end{equation}
on the unit square $[0,1]\times [0,1]$ with Dirichlet homogeneous boundary conditions. The number of inner grid points in both directions was $n_0$ and the dimension of matrices is $n=n_0^2$.

\noindent In table \ref{tab:table1}, we reported results for resGERA, resGEA and resGFOM. We used different values of the dimension $n$ ($\{2500,10000$ and $22500\}$) and two different block sizes $p=5,10$. The dimension of the subspace is chosen to be $m=10$ and $m=20$. As shown from this table, the resGFOM requires a higher number of restarts and cpu-time to reach convergence. Although the resGEA is able to reduce the number of restarts,  resGERA is much better in terms of number of restarts and cpu-time.    
\begin{table}[h!]
\caption{Example $1$: Shifted solvers for nonsymmetric matrices and different matrix dimensions for the operators given by \eqref{Op}  }
\label{tab:table1}
\begin{tabular}{c|c|c|c|c|c}
\hline
$Matrices.$ & $n$ & subspace & GERAM & GEAM & GFOM \\
 & & dimension & Time$(s)$(\#Cycles) & Time$(s)$(\#Cycles) & Time$(s)$(\#Cycles)\\
\hline
$A_1$ &  $2500$ & $10$ & $5.12$ $(2)$ & $9.65$ $(6)$ & $17.56$ $(47)$ \\ 
 & $2500$ & $20$ &$2.81$ $(1)$ &$7.90$ $(4)$ & $38.75$ $(28)$ \\
$s=5$ & $10000$ &$10$ & $8.17$ $(2)$ & $17.23$ $(8)$ & $156.67$ $(57)$\\
 & $10000$ &$20$ & $7.35$ $(1)$ &$27.47$ $(5)$ & $171.39$ $(32)$\\
 & $22500$ &$10$ & $10.43$ $(2)$ & $27.36$ $(9)$ & $558.76$ $(94)$\\
 & $22500$ & $20$ & $18.76$ $(1)$ & $20.86$ $(6)$ & $555.96$ $(45)$\\
 \hline
  $A_2$ & $2500$ & $10$ & $17.78$ $(2)$ & $10.57$ $(11)$ & $367.14$ $(91)$ \\ 
 & $2500$ & $20$ &$17.18$ $(1)$ &$20.37$ $(4)$ & $538.48$ $(45)$ \\
 $s=10$& $10000$ &$10$ & $16.69$ $(2)$ & $27.11$ $(6)$ & -\\
 & $10000$ &$20$ & $14.45$ $(1)$ & $20.61$ $(4)$ &-\\
 & $22500$ &$10$ & $35.10$ $(2)$ & $50.12$ $(5)$&-\\
 &$22500$ &$20$ & $34.09$ $(1)$  & $42.51$ $(3)$ &-\\
 \hline
\end{tabular}
\end{table}

\noindent {\bf Example 2} In this example, we used the nonsymmetric matrices $pde2961$, $epb1$, $add32$ and the symmetric matrix $mhd3200b$ from the Suite Sparse Matrix Collection \cite{DAVIS}. Some details on these matrices are given in Table \ref{tab:table_proprety}. Results for several choices of the block size $p$ are reported in  Table \ref{tab:table2}. The results show that the GERAM and GEAM yield significantly smaller cycles than GFOM. Moreover, the  GERAM is faster than GEAM for all matrices.

\begin{table}[h!]
\caption{Suite Sparse Matrix Collection matrices information}
\label{tab:table_proprety}
\renewcommand*{\arraystretch}{1.4}
\begin{tabular}{c|c|c|c|c|c|c}
\hline
Matrices & $Original\, Problem$ & size $n$ & $\lambda_{min}$& $\lambda_{max}$ & $cond(A)$ & $nnz$\\
\hline
$pde2961$ & economic problem & $2961$ & $0.04$ & $12.12$ & $6.42 \cdot 10^2$ & $14585$\\
$epb1$ & thermal problem & $14734$ & $4.85\times 10^{-5}$ & $15.66$ & $5940.66$ & $95053$\\
$mhd3200b$ & electromagnetics Problem & $3200$ & $1.36 \times 10^{-13}$ & $2.19$ &  $1.60 \times 10^{13}$ &$18316$\\
$add32$ &  circuit simulation problem & $4960$ & $4.21\times 10^{-4}$ & $0.06$ & $1.36 \cdot 10^2$ & $19848$\\
\hline
\end{tabular}
\end{table}

\begin{table}[h!]
\caption{Example 1: Shifted solvers for some matrices from the Suite Sparse Matrix Collection matrices}
\label{tab:table2}
\begin{tabular}{l|c|c|c|c}
\hline
Test problem & subspace & GERAM & GEAM & GFOM \\
 & dimension & Time$(s)$(\#Cycles) & Time$(s)$(\#Cycles) & Time$(s)$(\#Cycles)\\
 \hline
 $A_1=pde2961$ & $10$ & $9.07$ $(8)$ & $10.57$ $(11)$ & $20.72$ $(91)$ \\ 
 $n=2961$ & $20$ &$8.10$ $(2)$ &$9.90$ $(5)$ & $23.03$ $(25)$ \\
 $s=5$ & & & &\\
 \hline
  $A_2=epb1$ & $10$ & $32.34$ $(6)$ &  $56.02$ $(15)$& $402$ $(700)$ \\
  $n=14734,$  & $20$ & $37.63$ $(2)$ &$68.10$ $(7)$ &$206$ $(270)$ \\
  $s=10$ & & & &\\
 \hline
 $A_3=mhd3200b$ & $10$ & $10.28$ $(9)$ &  $52.03$ $(82)$& $321$ $(583)$ \\
  $n=3200$  & $20$ & $7.84$ $(3)$ &$38.7$ $(21)$ &$149.36$ $(134)$ \\
  $s=10$ & & & &\\
  \hline
  $A_4=add32$ & $10$ & $5.34$ $(3)$ &  $7.14$ $(5)$& $9.21$ $(24)$ \\
  $n=4960$  & $20$ & $5.12$ $(1)$ &$6.27$ $(3)$ &$10.82$ $(10)$ \\
  $s=5$ & & & &\\
  \hline
 \end{tabular}
\end{table}

\subsection{Examples for the apporixmation of $f(A)V$}
\noindent {\bf Example 3} In this example, we consider a semi discretization of the partial differential equation 
\begin{equation*}
\begin{array}{rll}
\dfrac{\partial U}{\partial t}-\Delta U+(x+y)\dfrac{\partial U}{\partial x}+(x-y)\dfrac{\partial U}{\partial y}&=0 & on\,  (0,1)^2\times (0,1)\\
U(x,y,t)&=0 & on \, \partial (0,1)^2 \,\forall t\in[0,1]\\
U(x,y,0)&=U_0(x,y)&  \forall x,y\in[0,1]^2.
\end{array}
\end{equation*}
where \begin{align*}
 U_0(x,y)&=\{u^{(1)}_0(x,y),u^{(2)}_0(x,y),u^{(3)}_0(x,y)\}\\
 &=\{\sin(\pi x)\sin(\pi y),\sin(2\pi x)\sin(\pi y),\sin(2\pi x)\sin(2\pi y)\},    
 \end{align*}
We used the nonsymmtric matrices $A_{100}$ and $A_{150}$ coming from CFDD of the operator 
\begin{align}\label{L3}
 \mathcal{L}_3(u)=-\Delta u + (x+y)u_x+ (x-y)u_y.
 \end{align}
 on the $[0,1]\times[0,1]$. The size of $A_{100}$ is $100\times 100$ and the size of $A_{150}$ is $150\times 150$. The subscript $100$ and $150$ denotes the number of inner grid points in both directions. The block $V$ is set to the values of the initial functions $U_0(x,y)$ on the finite-difference mesh $(x_i,y_j)$, with $x_i=(i-1)/(n_0-1)$ and $y_j=(j-1)/(n_0-1)$, for $i,j=1,\ldots,n_0$, i.e., $V(n_0(i-1)+j,k)=u_0^{(k)}(x_i,y_j)$, $k=1,2,3$. In this case, the block size is $p=3$. We computed approximations of $U(t)=e^{-tA}V$ correspond to the solution of partial differential equation. These approximations are given by the AGER method; see, Algorithm \ref{alg:exponential} and the adaptive rational Arnoldi method (ARA) described in \cite{Druskin}. We used different values of time parameters $t=\{1/10,1/3,2/3,1\}$. 
 The algorithms were stopped when residual norm $\|R_{2m}(t)\|$ is less than $5\times 10^{-9}$. \\ In table \ref{tab:table4}, we present results of this experiment. As shown in this table, the AGER method requires fewer iterations and CPU-time than ARA method. 
 \begin{table}[h!]
\caption{Example 3: Approximation of $e^{-tA}V$ for two matrix dimensions for the operator given by \eqref{L3}. }
\label{tab:table4}
\begin{tabular}{l|c|c|c|c|c|c}
\hline
\multirow{2}{2cm}{Test problem} & \multicolumn{3}{c|}{Adaptive global extended-rational Arnoldi method} &
\multicolumn{3}{c|}{Adaptive rational Arnoldi method}\\
\cline{2-7}
& Sp. dimen. & Res. norm & Time(s) &
Sp. dimen. & Res. norm & Time(s)\\
\hline
$A_{100}$ & & & & & &  \\
$t=1/10$ & 50 & $2.15\,\times\,10^{-9}$ &
  5.77 & 100 &  $1.06\,\times\,10^{-9}$ & 108.09\\
$t=1/3$ & 40 & $5.85\,\times\,10^{-9}$ &
  4.72 & 95 &  $1.14\,\times\,10^{-9}$ & 83.20\\
$t=2/3$ & 28 & $1.19\,\times\,10^{-9}$ &
  2.98 & 60 &  $1.98\,\times\,10^{-9}$ & 24.53\\
$t=1$ & 16 & $1.94\,\times\,10^{-9}$ &
  2.13 & 32 &  $2.22\,\times\,10^{-9}$ & 10.06\\
\hline
$A_{150}$ & & & & & &  \\
$t=1/10$ & 54 & $3.26\,\times\,10^{-9}$ &
  13.48 & 100 &  $7.00\,\times\,10^{-7}$ & 275.16\\
$t=1/3$ & 46 & $3.77\,\times\,10^{-9}$ &
  11.37 & 100 &  $2.06\,\times\,10^{-6}$ & 274.45\\
$t=2/3$ & 30 & $1.87\,\times\,10^{-9}$ &
  7.72 & 96 &  $5.82\,\times\,10^{-9}$ & 260.12\\
$t=1$ & 30 & $1.29\,\times\,10^{-9}$ &
  6.02 & 50 &  $3.73\,\times\,10^{-9}$ & 56.71\\
  \hline
\end{tabular}
\end{table}

\noindent In the following examples, we compare the performance of GERA method with the performance of the rational arnoldi (RA) method and the standard global Arnoldi (SGA) method. In all examples, $A\in\mathbb{R}^{1000\times 1000}$, and the block $V\in\mathbb{R}^{n\times 5}$ was generated randomly with entries uniformly distributed on $[0,1]$. The dimension of the Krylov subspace is chosen $m=20$. We determine the actual value $\mathcal{I}(f)$ given by \eqref{If} using funm function in MATLAB. In the tables, we display the errors $Er(f^{SGA}_m)=\|\mathcal{I}(f)-f^{SA}_m\|$, for the SA method $Er(f^{ER}_{m/2})=\|\mathcal{I}(f)-f^{ER}_{m/2}\|$ for the GERA method and $Er(f^{RA}_m)=\|\mathcal{I}(f)-f^{RA}_m\|$ for the RA method. In the extended-rational method, the poles are chosen as $s_i=0.1i$ for $i=1,\ldots,10$, while in the rational method the poles are chosen as $s_i=0.05i$ for $i=1,\ldots,20$. 

\noindent {\bf Example 4}  
Let $A=[a_{i,j}]$ be the symmetric positive definite Toeplitz matrix with entries $a_{i,j}=1/(1+|i+j|)$ \cite{Jagels}. Results for several functions are reported in Table \ref{tab:table5}. As shown, the approximations computed with the GERA method are more accurate than approximations determined by the RA and SGA methods.
\begin{table}[h!]
\caption{Example 4: $A\in\mathbb{R}^{n\times n}$ is a symmetric positive definite Toeplitz matrix with $n=1000$. Block size $p=5$.} \label{tab:table5}
\centering
\begin{tabular}{l|c|c|c}
\hline
 $f(x)$ & $Er(f^{ER}_{m/2})$ & $Er(f^{AR}_{m})$ & $Er(f^{SA}_{m})$ \\
\hline
$\sqrt{x}$ & $1.47\times 10^{-12}$ & $3.69\times 10^{-7}$&$2.44\times 10^{-5}$ \\
$\ln{x}$ & $3.38\times 10^{-12}$ &  $1.13\times 10^{-7}$& $2.84\times 10^{-4}$ \\
$exp(-\sqrt{x})$ &$1.79\times 10^{-11}$ &$2.12\times 10^{-8}$ & $2.38\times 10^{-4}$\\
 \hline
\end{tabular}
\end{table}

\noindent {\bf Example 5}
The matrix used in this example is a block diagonal with $2\times2$ blocks of the form
\[
\begin{bmatrix}
a_i & & c\\
-c &  &a_i
\end{bmatrix}
\]
where $c=1/2$ and $a_i=(2i-1)/(n+1)$ for $i=1,\ldots,n/2$ \cite{SAAD1}. Table $\ref{tab:table6}$ displays computed results, and shows that approximations computed with the GERA method have higher accuracy than approximations obtained by the RA and SGA methods.
\begin{table}[h!]
\caption{Example 5: $A\in\mathbb{R}^{n\times n}$ is a block diagonal matrix with $2\times 2$ blocks. $n=1000$ and block size $p=5$.}
\label{tab:table6}
\centering
\begin{tabular}{l|c|c|c}
\hline
 $f(x)$ & $Er(f^{ER}_{m/2})$ & $Er(f^{AR}_{m})$ & $Er(f^{SA}_{m})$ \\
\hline
$\sqrt{x}$ & $2.99\times 10^{-10}$ & $1.26\times 10^{-7}$&$5.64\times 10^{-4}$ \\
$\ln{x}$ & $7.04\times 10^{-10}$ &  $4.54\times 10^{-9}$& $8.2\times 10^{-3}$ \\
$exp(-\sqrt{x})$ &$5.38\times 10^{-10}$ &$4.53\times 10^{-9}$ & $5.56\times 10^{-4}$\\
 \hline
\end{tabular}
\end{table}
\section{Conclusion}\label{sec5}
This paper describes the global extended-rational Arnoldi method for the approximation of $f(A)V$ and for solving parameter dependent systems \eqref{SL}. We proposed an adaptive procedure to compute the shifts when $f(A)=e^{-tA}$ or $f(A)=(A-\sigma I_n)^{-1}$. The numerical results show that the proposed algorithms AGER (resGERA) require fewer iterations (number of restarts) and cpu-time as compared to other projection-type methods when approximating $f(A)V$  and when solving parameter dependent systems.


\end{document}